\documentclass[12pt]{article}
\usepackage[english]{babel}
\usepackage{graphicx} 
\usepackage{latexsym,color,amsmath,amsthm,amssymb,amscd,amsfonts}
\usepackage{scalefnt}
\usepackage[font=small,labelfont=bf]{caption}
\usepackage{float}
\usepackage{graphicx}
\usepackage{amsmath}
\usepackage{relsize}
\usepackage{mathrsfs}
\usepackage{upgreek}
\usepackage{sectsty}
\usepackage{bbm}

\newtheoremstyle{nonum}{}{}{\itshape}{}{\bfseries}{.}{ }{\thmnote{#3}}

\makeatletter
\newcommand{\cbigoplus}{\DOTSB\cbigoplus@\slimits@}
\newcommand{\cbigoplus@}{\mathop{\widehat{\bigoplus}}}
\makeatother

\newtheorem{thm}{Theorem}[section]
\newtheorem*{thm*}{Theorem}

\newtheorem{cor}[thm]{Corollary}
\newtheorem{lem}[thm]{Lemma}
\newtheorem*{lem*}{Lemma}
\newtheorem{rem}[thm]{Remark}
\newtheorem*{rem*}{Remark}
\newtheorem{prop}[thm]{Proposition}
\newtheorem*{prop*}{Proposition}

\newtheorem*{definition*}{Definition}

\newtheorem*{fact*}{Fact}

\newtheorem*{rems*}{Remarks}

\theoremstyle{nonum}

\newcommand{\R}{\mathbb R}

\newcommand{\Z}{\mathbb Z}
\newcommand{\N}{\mathbb N}
\newcommand{\Q}{\mathbb Q}

\usepackage{bbm}

\def\conv{{\rm conv}}

\DeclareMathOperator{\Int}{Int}

\usepackage{xcolor}
\usepackage{hyperref}

\hypersetup{
  colorlinks=true,
  linkcolor=blue,
  citecolor=blue,
  urlcolor=blue
}

\title{\normalsize\textbf{\MakeUppercase{Asymptotics For Face Numbers of Certain Hanner Polytopes, With Applications}}}
\author{Tomer Milo}
\date{}

\begin{document}

\maketitle

\begin{abstract}
   \noindent We provide asymptotics for the number of faces of a certain family of Hanner polytopes. As a corollary, we come close to saturating the FLM inequality for a certain family of parameters.
\end{abstract}

\section{Introduction}
This note provides asymptotics for the face numbers of a certain family of Hanner polytopes which naturally appears in the study of the Figiel-Lindenstrauss-Milman (FLM) inequality, originally proved in \cite{FLM}. It states the following.
\begin{thm}\label{thm: FLM}
    Let $P \subset \R^n$ be a polytope. Let $V$, $\mathcal{F}$ denote the set vertices $(0-\text{dimensional faces})$ and facets $((n-1)-\text{dimensional faces})$ respectively. Let $B_2^n$ denote the Euclidean unit ball, and assume $rB_2^n \subset P \subset RB_2^n$. Then
    \[ \log|V| \cdot \log|\mathcal{F}| \cdot \left(\frac{R}{r}\right)^2 \geq c n^2 \]
    where $c>0$ is some universal constant.
\end{thm}
\noindent In \cite{Mil} the question of tightness of this inequality has been explored. The following was proven (here and throughout, $A \sim B$ means that $A = \Theta(B)$, that is, the two quantities are equivalent up to a universal constant).
\begin{thm}\label{thm: my FLM stuff}
\text{1.} For any $a \in (0,1)$, There exists a sequence of polytopes $P_n \subset \R^n$ such that 
\[ \log|V_n| \sim n^a, \: \log|\mathcal{F}_n| \sim n^{1-a}, \quad \text{and} \quad \left(\frac{R_n}{r_n}\right)^2 = n \]

\noindent \text{2.} For any $a \in (0,1)$ There exists a sequence of polytopes $P_n \subset \R^n$ such that
        \[ \log|\mathcal{F}_n| \sim n^{1-a}, \quad \log|V_n| \sim n \:, \quad \text{and} \quad \left( \frac{R_n}{r_n} \right)^2 \sim n^a \]

\noindent \text{3.}  For any $a, \delta \in (0,1)$ There exists a sequence of polytopes $P_n \subset \R^n$ such that
\[ \log|\mathcal{F}_n| = O(n^{1-a}), \quad \log|V_n| = O(n^{a+\delta}) \:, \quad \text{and} \quad \left( \frac{R_n}{r_n} \right)^2 =O( n^{1-\delta + a}), \]
so that 
\[ \log|V_n| \cdot \log|\mathcal{F}_n| \cdot \left( \frac{R_n}{r_n} \right)^2 = O( n^{2+a}). \]
\end{thm}
\noindent Parts 1 and 2 provides families of polytopes which completely saturate the FLM inequality for a certain family of parameters; namely, $\{(a,1-a,1): a \in (0,1)\}$ for the first part, and $\{(1-a,1,a): a \in (0,1)\}$ for the second part.  \\
This note slightly improves 3. We prove:
\begin{thm}\label{thm: main FLM improvment}
    Let $a, \delta \in (0,1)$. 
    
\noindent 1. If $ a \in \Q$, there exists a sequence of polytopes $P_{n} \subset \R^{n}$ such that
\[ \log|\mathcal{F}_{n}| = O(n^{1-a}), \quad \log|V_{n}| = O(n^{a+\delta(1-a)})\:, \quad \text{and} \quad \left( \frac{R_{n}}{r_{n}} \right)^2 = O(n^{1-\delta + a} )\]
so that 
\[ \log|V_{n}| \cdot \log|\mathcal{F}_{n}| \cdot \left( \frac{R_{n}}{r_{n}} \right)^2 = O( n^{2+a(1-\delta) }) \]
\smallskip
\noindent 2. If $a \notin \Q$, there exists a sequence of polytopes $P_n \subset \R^n$ such that
\[ \log|\mathcal{F}_{n}| = O(n^{1-a}), \quad \log|V_{n}| = O(n^{a+\delta(1-a) + o_n(1)})\:, \quad \text{and} \quad \left( \frac{R_{n}}{r_{n}} \right)^2 = O(n^{1-\delta + a} ) \]
so that
\[ \log|V_{n}| \cdot \log|\mathcal{F}_{n}| \cdot \left( \frac{R_{n}}{r_{n}} \right)^2 = O( n^{2+a(1-\delta) + o_n(1) })  \]
Concretely, we may take $o_n(1)$ to be $ O(\frac{1}{\sqrt{\log n}})$ in the case $a \notin \Q$. All the implied constants in the big-$O$ notation depend only on $a$ and $\delta$. 
\end{thm}
\begin{rem}\label{remark: extending to all dimensions}
    Our proof only provides a sequence of polytope $P_{n} \subset \R^{d_n}$ where $d_n \sim 2^n$. One can extend this sequence to a family of polytopes in any dimension with the same asymptotics by standard techniques (see the proof of Proposition 4.1 in \cite{Mil}).
\end{rem}

Our basic family of polytopes (referred to as the \emph{basic polytopes} from now on), which was presented in \cite{Mil}, is the following. Fix $a \in (0,1)$. Let $A_a = \{ \lfloor\frac{n}{a} \rfloor : n \in \N\} $, so that the density of $A_a$ in $\N$ is $a$. We define a sequence of polytopes $P_{n}^a \subset \R^{2^n}$ inductively: Let $P_0 = [-1,1]$, and given $P_{n-1}^a \subset \R^{2^{n-1}}$ we let
\[ P_{n}^a = \begin{cases} P_{n-1}^a \times P_{n-1}^a, \: n \in A_a \\
\conv(P_{n-1}^a, \: P_{n-1}^a), \: n \notin A_a \\
\end{cases} \]
where in each step, we orthogonally decompose $\R^{2^n} = \R^{2^{n-1}} \times \R^{2^{n-1}}$, and take two copies of $P_{n-1}^a$, one in each copy of $\R^{2^{n-1}}$, and perform the operation (product or convex hull) on these two copies. \\
The basic polytopes satisfy part $1$ of Theorem \ref{thm: my FLM stuff}. The basic polytopes can be `extended' from dimensions that are powers of $2$ to all dimensions; see Proposition 4.1 in \cite{Mil} for details. Parts 2 and 3 are satisfied by sections of the basic polytopes, as will be shortly explained. 
\begin{rem}
    The basic polytopes $P_n^a$ can be thought of as a continuous interpolation between the cube and the $l_1$ ball -- indeed, $a=0$ gives the $l_1$ ball, while for $a=1$ we get the cube.
\end{rem}

If $P \subset \R^d$ is $d$-dimensional, we denote by $f_{k}(P)$ the number of $k$ dimensional faces of $P$ (so $k \in \{0,1,\ldots,d-1\}$, $f_{0}(P) = |V|$ and $f_{d-1}(P) = |\mathcal{F}|$). \\
A basic fact which will be repeatedly used in this note is the following: If $E \subset \R^d$ is a generic $(d-k)$-dimensional subspace, then 
\begin{equation}\label{eq1} f_{k}(P \cap E) \leq f_{d-k}(P) \end{equation}
That is to say, $k$-dimensional faces of a generic section can only arise from $d-k$ dimensional faces of $P$.\\
In (\cite{Mil}, Proposition 4.6) it was shown that random sections of $P_n^a$ of dimension $\frac{d}{2}$ satisfy part 2 of Theorem \ref{thm: my FLM stuff} (with high probability). Part 3 of Theorem \ref{thm: my FLM stuff} was obtained by taking $d-  \lfloor d^{\delta}\rfloor $-dimensional sections of the basic polytopes. Using the crude bound $f_{k}(P) \leq \binom{f_0(P)}{k+1}$, the bound $O(d^{a+\delta})$ on the logarithm of number of vertices of the section was obtained by (\ref{eq1}) via the crude upper bound
\[ \log f_{\lfloor d^{\delta} \rfloor}(P_n^a) \leq \log \binom{f_0(P_n^a)}{\log \lfloor d^{\delta}\rfloor +1} = O (d^{\delta +a}) \]
Theorem \ref{thm: main FLM improvment} improves part 3 of Theorem \ref{thm: my FLM stuff} by providing precise asymptotics for $\log_{\lfloor d^{\delta} \rfloor}f(P_n^a)$ instead of using the crude bound. By the discussion above, Theorem \ref{thm: main FLM improvment} will be an immediate corollary of the following

\begin{thm}\label{thm: Hanner asymptotics}
    Let $P_n^a \subset \mathbb{R}^{2^n}$ be the basic polytopes. Let $d = 2^n$ denote the dimension. For all $\delta \in (0,1)$,
    
 \noindent   1. If $a \in \Q$, then 
    \[ \log f_{\lfloor d^{\delta}\rfloor}(P_n^a) \sim d^{a + \delta(1-a)}. \]

    \noindent 2. If a $\notin \Q$, then 
    \[ \log f_{\lfloor d^{\delta}\rfloor}(P_n^a) = d^{a + \delta(1-a) + o_d(1)}. \]
Concretely, we may take the $o_d(1)$ term to be $O(\frac{1}{\sqrt{\log d}})$.
All the implied constants in both of the statements depend only on $a$ and $\delta$. 
\end{thm}

\section{Face numbers as a recursion sequence}
For a fixed $a  \in (0,1) $, let $P_n^a \subset \R^{d}$ be the basic polytopes corresponding to $a$, where $d = d_n = 2^n$. We let
\[ a_{n,k} = f_{k}(P_n^a), \: k = 0,1,\ldots,d-1 \]
Note that $a_{0,0} = 2$ and $a_{0,1} = 1$, since a segment has $2$ vertices and one facet (edge).
From the definition of the basic polytopes, the sequence $a_{n,k}$ satisfies the following recursion. 
\begin{lem} for all $n,k \geq 0$,
    \[ a_{n+1,k} = \begin{cases} \sum_{j=0}^k a_{n,j}a_{n,k-j}, \: n \in A_a \\ 2a_{n,k} +\sum_{j=0}^{\min\{d-1,k-1\}} a_{n,j}a_{n,(k-1)-j}, \: n \notin A_a
    \end{cases} \]
where $a_{n,d} = a_{n,-1} := 1 $ for all $n \geq 0$, and $a_{n,k} :=0$ whenever $k < -1 $ or $k > d$.
\end{lem}
The case $n \in A_a$ (product step) says that $k$-dimensional faces of a product are obtained by multiplying $j$-dimensional faces with $k-j$ dimensional faces. The case $n \notin A_a$ (convex hull step) says that a similar thing happens, only now the dimensions sum to $k-1$; the extra $2a_{n,k}$ term follows form the fact that a $k$-dimensional face of one polytope remains a $k$-dimensional face after the convex hull operation. 
\\
 In the convex hull case, the sum goes up to $\min\{k-1,d-1\}$ so that we have $a_{n+1,2d-1} = a_{n,d-1}^2$, but since in this note we only consider $k = d^{\delta}$ for $\delta \in (0,1)$, we may ignore that technical requirement, and indeed we will assume from now on that $k \leq d$.
 \\
\noindent The following remark collects two easy properties of the sequence $a_{n,k}$, for future use.

\begin{rem}\label{remark: increasing sequence}
    1. From the definition of $a_{n,k}$, we have
    \[ a_{n+1,k} \geq a_{n,k} \]
    for all $n,k$.
    \\

    \noindent 2. Let $A_{n,k} = \max_{j \leq k} a_{n,j}$. Simple induction shows that
    \[ a_{n+r,k} \leq k^{2^{r}} A_{n,k}^{2^r} \]
\end{rem}
\smallskip

Let $F_n(t) = \sum_{k \geq 0} a_{n,k} t^k$ be the generating function of the sequence $a_{n,k}$. Clearly, $F_n$ satisfies the following recursion:
\[
    F_0(t) = 2+t, \quad F_{n+1}(t) = \begin{cases} F_n(t)^2, \: n \in A_a \\ tF_n(t)^2 + 2F_n(t), \: n \notin A_a
    \end{cases}
\]
We fix some notation that will be used throughout the proof of Theorem \ref{thm: Hanner asymptotics}. Let $Q$ be a positive integer and let $m = \lfloor \frac{n}{Q} \rfloor$ (we reveal to the reader in advance that in the case that $a=\frac{p}{q} \in \Q$ we will choose $Q = q$, and in the case $a \not\in \Q$ we will choose $Q=Q(n) = \sqrt{n}$). Let
\[ H_m(t) = F_{Qm}(t). \]
Note that $a_{Qm,k}$ (the coefficient of $t^k$ in the expansion of $H_m$) is $f_{k}(P_{Qm}^a)$. \\

\noindent Let $R(x,t) = tx^2 +2x$ and $S(x) = S(x,t) = x^2$. Then, we may write
\begin{equation}\label{eq2} H_{m+1}(t) = \phi_{Q,m}(H_m(t),t),\end{equation}
where 
\[ \phi_{Q,m}(x) = T_{Qm+Q-1} \circ T_{Qm+Q-2} \circ \ldots\circ T_{Qm} \:, \]
where $T_{j}$ is either $S$ or $R$, depending on whether $j \in A_a$ or not. Let \\ $p(Q,m) \in \{ \lfloor aQ\rfloor, \lceil aQ \rceil \}$ denote the number of applications of $S$ in the definition $\phi_{Q,m}$. The following Lemma is straightforward to prove -- it collects some basic facts about $\phi_{Q,m}$ we will repeatedly use.
\begin{lem}\label{lemma: facts about phi}
 There exists some finite $K \subset \Z_{>0}$ such that 
        \[ \phi_{Q,m}(x,t) = \sum_{k \in K}C_k(t)x^k, \]
where $C_k(t)$ are of the form $A_kt^{\lambda_k}$ for $A_k,\lambda_k \in \Z_{\geq 0}$. We have $\deg(\phi_{Q,m}) = 2^Q$, and $\lambda_k =0$ if and only if $k = 2^{p(Q,m)}$ (this is the monomial obtained by choosing the $t$-free term at each iteration).
\end{lem}

\section{Solving the recursion for $H_m$ using trees}
The goal of this section is to solve the recursion (\ref{eq2}) for $H_m$. Let $\mathcal{T}_m^K$ denote the set of trees of height exactly $m$ (so that all leaves are of height $m$), and with each internal vertex of degree in $K$. For $T \in \mathcal{T}_m^K
$, let $L(T)$ denote the set of leaves of $T$, let $\Int(T)$ denote the set of internal vertices (i.e, non-leaves) of $T$, and let $\deg(v)$ denote the degree of the vertex $v$. Define the \emph{weight} of a tree to be
\[ W(T) := \prod_{v \in \Int(T)} C_{\deg(v)}(t) \]
with the convention $\Int(T) = \emptyset \implies W(T)=1$.
\begin{prop}\label{prop: H_n trees}
    For any $m \geq 0$, 
    \[ H_m(t) = \sum_{T \in \mathcal{T}_m^K} W(T)(2+t)^{L(T)}. \]
\end{prop}
\begin{proof}
    The proof is by induction. The case $m=0$ is trivial. Assume the equality holds for some $m$. By definition,
    \[ H_{m+1}(t) = \phi(H_m(t),t) = \sum_{k \in K} C_k(t)H_m(t)^k. \]
    For a fixed $k\in K$, by the induction hypothesis
    \[ H_m(t)^k =  \left( \sum_{T \in \mathcal{T}_m^K} W(T) (2+t)^{L(T)} \right)^k = \sum_{(T_1,\ldots,T_k) \in (\mathcal{T}^K_n)^k} \prod_{i=1}^k W(T_i)(2+t)^{L(T_i)}.  \]
So,
\[ H_{m+1}(t) = \sum_{k \in K} C_k(t) \left( \sum_{(T_1,\ldots,T_k) \in (\mathcal{T}_m^K)^k} \prod_{i=1}^k W(T_i)(2+t)^{L(T_i)} \right). \]
Now, let $T \in \mathcal{T}_{m+1}^K$ be such that $\deg(root_T) = k$. Let $T_1,\ldots,T_k \in \mathcal{T}_m^K$ be the height-$m$ trees starting at each of the $k$ children of the root of $T$. Then, $L(T) = \sum_{i=1}^k L(T_i)$, and $\Int(T) = 1 + \sum_{i=1}^k \Int(T_i)$. Therefore
\[ W(T)(2+t)^{L(T)} = C_k(t) \prod_{i=1}^k W(T_i) (2+t)^{L(T_i)}. \]
We conclude
\[ \sum_{T \in \mathcal{T}_{m+1}^K} W(T)(2+t)^{L(T)} = \sum_{k \in K} \sum_{\substack{T \in \mathcal{T}_{m+1}^K \\ \deg(root_T) = k}} C_k(T) \prod_{i=1}^k W(T_i)(2+t)^{L(T_i)} = \]
\[ = \sum_{k \in K} \sum_{\substack{T \in \mathcal{T}_{m+1}^K \\ \deg(root_T) = k}} W(T)(2+t)^{L(T)} = \sum_{T \in \mathcal{T}_{m+1}^K} W(T) (2+t)^{L(T)} \]

\end{proof}

\begin{cor}\label{cor: coefficient formula} We have the following expression for the coefficients of $H_m$:
    \[ a_{Qm,k} = [t^k]H_m(t) = \sum_{j=0}^k \sum_{T \in \mathcal{T}_m^K } [t^j]W(T) \cdot \binom{L(T)}{k-j}2^{L(T)-(k-j)}  \]

\end{cor}

\noindent Our goal now (in order to prove Theorem \ref{thm: Hanner asymptotics}) is to estimate $a_{Qm,k}$. The lower and upper bounds would require different techniques. We start with
 \section{The Upper Bound}
 The goal of this section is to use the expression obtained in Corollary \ref{cor: coefficient formula} to upper bound $a_{Qm,k}$. We prove
 \begin{prop}\label{prop: upper bound coefficients}
     Let $a_{Qm,k}$ be as in Corollary \ref{cor: coefficient formula}. Then
     \[ \log a_{Qm,k} = O_a(2^{O(Q)}2^{np(Q,m)}k^{1-\frac{p(Q,m)}{Q}}) \]
 \end{prop}
\noindent This Proposition requires several Lemmas, starting with
\begin{lem}\label{lem: Q(T)}
    In the expression for $a_{Qm,k}$ obtained in Corollary \ref{cor: coefficient formula}, the only trees $T \in \mathcal{T}_m^K$ contributing non-trivially have
    \[ Q(T) := |\{v: \deg(v) \neq 2^{p(Q,m)}\}| \leq k.\]
\end{lem}
\begin{proof}
    As explained in Lemma \ref{lemma: facts about phi}, for all $k \neq 2^{p(Q,m)}$ we have $\deg(C_k(t)) \geq 1$. But 
    \[ \deg(W(T)) = \sum_{v \in \Int(T)} \deg(C_{\deg(v)}) = \sum_{ \substack{v \in \Int(T) \\ \deg(v) \neq 2^{Q}}} \deg(C_{\deg(v)}). \]
and if $Q(T) >k$, this last sum is $>k$ as well. By Corollary \ref{cor: coefficient formula}, this $T$ cannot contribute to $a_{Qm,k}$.
\end{proof}
 \begin{lem}\label{lemma: W(T)(1) bound} For all $T \in \mathcal{T}_m^K$,
    \[ \log W(T)(1) = O(2^Q|\Int(T)|) \] 
\end{lem}
\begin{proof}
    Let $C := \max_{k \in K}C_k (1) \leq 2^{2^Q}$. Then
    \[\log W(T)(1) = \log\prod_{v \in \Int(T)}C_{\deg(v)}(1) \leq  \log  \prod_{v \in \Int(T)} 2^{2^Q}  = O(2^Q|\Int(T)|).\]
\end{proof}

\begin{lem}\label{lemma: L(T) upper bound} For all $T \in \mathcal{T}_m^K$ with $Q(T) \leq k$,
    \[L(T) = O(2^{O(Q)}2^{np(Q,m)}k^{1-\frac{p(Q,m)}{Q}}), \]
where the implied constant depends only on $a$.
\end{lem}
\begin{proof} Let $T \in \mathcal{T}_m^K$.
    Let $N_j$ denote the number of vertices of $T$ at height $j$, so that $L(T) = N_n$. Let $Q_j$ denote the number of internal vertices of height $j$ of degree which is \emph{not} $2^{p(Q,m)}$. Then, $\sum_{j=0}^{n-1}Q_j = Q(T) \leq k$, by Lemma \ref{lem: Q(T)}.\\
    We compare $N_{j+1}$ to $N_j$. every vertex at height $j$ has either degree $2^{p(Q,m)}$, or degree $\neq 2^{p(Q,m})$ (but at most $2^Q$). So
    \begin{equation}\label{eq3} N_{j+1} \leq 2^{p(Q,m)}(N_j - Q_j) + 2^Q Q_j = 2^{p(Q,m)} N_j + Q_j(2^Q - 2^{p(Q,m)}) \leq 2^{p(Q,m)}N_j + 2^Q Q_j. \end{equation}
    Let $h := \lfloor \log_{2^Q}k\rfloor$. Trivially, $N_h \leq (2^Q)^h \leq k$. Iterating (\ref{eq3}) from $j = h$ to $n-1$ gives
    \[L(T) =  N_n \leq (2^{p(Q,m)})^{n-h}N_h + 2^Q\sum_{i=h}^{n-1}Q_i (2^p)^{n-1-i}  = O_a( 2^{O(Q)}2^{np(Q,m)}k^{1-\frac{p(Q,m)}{Q}}). \]

\end{proof}
\noindent The final Lemma required for Proposition \ref{prop: upper bound coefficients} says that $|\{ T \in \mathcal{T}_m^K: Q(T) \leq k\}|$ is not too large.
\begin{lem}\label{lemma: few trees}
    \[|\{ T \in \mathcal{T}_m^K: Q(T) \leq k\}| \leq (|K|+1)^{|V(T)|}\]
Where $V(T)$ is the set of vertices of $T$ (in our case, $|K| \leq 2^{Q}$).
\end{lem}
\begin{proof}
    Let $\Sigma = \{0,1,\ldots,|K|\}$. List the vertices in preorder: visit the root, then recursively visit the children from left to right. Define the word 
    \[ w(T) = d_1d_2 \ldots d_{|V(T)|} \in \Sigma^{|V(T)|}, \]
    clearly this mapping is an injection, which proves the lemma.
\end{proof}
\noindent We now prove Proposition \ref{prop: upper bound coefficients}.
\begin{proof}[Proof of Proposition \ref{prop: upper bound coefficients}] Using the representation for $a_{Qm,k}$ from Corollary \ref{cor: coefficient formula}, we may bound
 \[ a_{Qm,k} \leq \sum_{j=0}^k \sum_{T \in \mathcal{T}_m^K} W(T)(1) \binom{L(T)}{k-j}2^{L(T)-(k-j)} \leq k\sum_{T \in \mathcal{T}_m^K} W(T)(1) 4^{L(T)}, \]
where for every $j$, instead of taking $[t^j]W(T)$, we simply took the sum of \emph{all} the coefficients together, i.e, we take $W(T)(1)$, and we also bounded both the binomial coefficients and $2^{L(T)-(k-j)}$ by $2^{L(T)}$. By Lemmas \ref{lemma: W(T)(1) bound}, \ref{lem: Q(T)} \ref{lemma: L(T) upper bound}, and \ref{lemma: few trees}, we get
\[ \log a_{qn, k} = O(\log k + 2^{Q}|\Int(T)| + 2^{Q}|V(T)|+ L(T)).  \]
Finally, noting that $L(T) = 1+ \sum_{v \in \Int(T)}(\deg(v)-1)$, and that every internal vertex in our graph is of degree at least $2$, we get $|\Int(T)| \leq L(T) - 1 $, and so \\ $|V(T)| = |L(T)| + |\Int(T)| \leq 2L(T) =O(2^{O(Q)}2^{np(Q,m)}k^{1-\frac{p(Q,m)}{Q}})$. We are done.
\end{proof}

\section{The Lower Bound}

We are going to bound $a_{Qm,k}$ from below using the representation from Corollary \ref{cor: coefficient formula}, by choosing a special family of trees $T_{m} \in \mathcal{T}_m$ with many leaves, which will, informally, satisfy
\[ \log a_{Qm,k} = \Omega( L(T_m)). \]
To be precise, we prove
\begin{prop}\label{prop: lower bound} Let $a_{Qm,k}$ be as in Corollary \ref{cor: coefficient formula}. If $k = d^{\delta} = 2^{\delta m Q}$ for some $\delta \in (0,1)$, then, for large $m$,
\[ \log a_{Qm,k} = \Omega(2^{O(Q)}2^{mp(Q,m)}k^{1-\frac{p(Q,m)}{Q}}), \]
where the implied constant depends only on $a$ and $\delta$.
\end{prop}
We build a family of trees $T_m \in \mathcal{T}_m^K$ as follows. As explained in Lemma \ref{lemma: facts about phi}, we can write $\phi_{Q,m}(x,t) = Ax^{2^{p(Q,m)}} + Bt^{\lambda}x^{2^Q} + (\text{other terms})$, for some \emph{positive} integers $A,B,\lambda$. start with the full $2^Q$-ary tree (where each vertex has degree $2^Q$) up to height $h-1$, for $h := \lfloor \log_{2^Q}(\frac{k}{2\lambda})\rfloor $. For heights $h,h+1,\ldots,m-1$, every vertex has degree $2^{p(Q,m)}$; that is, each of the $(2^Q)^h$ vertices with height $h-1$ is a root of a full $2^{p(Q,m)}$-ary tree of height $m-h$. Note that the number of vertices of degree $2^Q$ is
    \[ \sum_{i=0}^{h-1} (2^Q)^i = \frac{(2^Q)^h - 1}{2^Q - 1} \leq (2^Q)^h \leq k \]
    i.e, the trees $T_m$ satisfy $Q(T_m) \leq k$. Let $j^* := \lambda (2^Q)^h \leq \frac{k}{2}$.
\begin{lem}\label{lemma: W(T_m) and j^*}
    $[t^{j^*}]W(T_m) \geq 1 $.
\end{lem}
\begin{proof}
    \[
        W(T_m)(t) = \prod_{v\in V(T): \deg(v) = 2^{p(Q,m)}} C_{2^{p(Q,m)}}(t) \prod_{v \in V(T): \deg(v) = 2^Q} C_{2^Q}(t) = \] 
        \[ \prod_{v\in V(T): \deg(v) = 2^{p(Q,m)}} A \prod_{v \in V(T): \deg(v) = 2^Q} Bt^{\lambda}
    \]
Since there are exactly $(2^Q)^h$ vertices of degree $2^Q$, the lemma follows.
\end{proof}
\begin{lem}\label{lemma: L(T_m) lower bound} We have
    \[ L(T_m) = \Omega_{a,\delta}(2^{O(Q)}2^{p(Q,m)m}k^{1-\frac{p(Q,m)}{Q}}) \]
Moreover, if $k = d^{\delta} = 2^{\delta Qm}$ for some fixed $\delta \in (0,1)$, then $L(T_m) > 2k$ for large $m$.
\end{lem}
\begin{proof}
    \[ L(T_m) = (2^Q)^h \cdot (2^{p(Q,m)})^{m-h} = 2^{Qh}\cdot 2^{p(Q,m)m}\cdot 2^{-hp(Q,m)} \geq  \]
\[
    (2^Q)^{\log_{2^Q}(\frac{k}{2\lambda})-1}\cdot \left((2^{Q})^{-\log_{2^Q}(\frac{k}{2\lambda})} \right)^{\frac{p(Q,m)}{Q}}\cdot (2^{mQ})^{\frac{p(Q,m)}{Q}} = 2^{-Q}\frac{k}{2\lambda} \cdot (\frac{k}{2\lambda})^{-\frac{p(Q,m)}{Q}} \cdot (2^{mQ})^{\frac{p(Q,m)}{Q}} = 
\]
\begin{equation}\label{eq4} 2^{-Q}\left( \frac{1}{2\lambda}\right)^{1-\frac{p(Q,m)}{Q}} k^{1-\frac{p(Q,m)}{Q}}(2^{mQ})^{\frac{p(Q,m)}{Q}} = 2^{O(Q)} k^{1-\frac{p(Q,m)}{Q}}(2^{mQ})^{\frac{p(Q,m)}{Q}} \end{equation}

\noindent We now prove that $L(T_m) > 2k$ for large $m$. Moving stuff around in the lower bound we got in (\ref{eq4}), we see that for this to happen, it suffices that
\[ 2^{mQ} > 2^{ O(Q)}  2^{\delta m Q} \]
Equivalently,
\[ 2^{(1-\delta)mQ} > 2^{O(Q)} \]
Which clearly holds for large $m$, as long as $Q$ is not too big (recalling that $m = \lfloor \frac{n}{Q} \rfloor$ we may take $Q = o(n)$,  as we will in the proof of Theorem \ref{thm: Hanner asymptotics}).
\end{proof}
\noindent As an immediate Corollary, we have
\begin{cor}\label{cor: binom L(T_m)}
    \[\binom{L(T_m)}{k -j^*} \geq 1. \]
\end{cor}

\begin{proof}[Proof of Proposition \ref{prop: lower bound}]
From Corollary \ref{cor: coefficient formula} and Lemma \ref{lem: Q(T)}, we have, for all $T \in \mathcal{T}_m^K$ with $Q(T) \leq k$:
\begin{equation}\label{eq5}
     a_{Qm,k} \geq \sum_{j=0}^k [t^j]W(T) \binom{L(T)}{k-j}2^{L(T)-(k-j)}.
\end{equation}
\noindent We use Lemma \ref{lemma: W(T_m) and j^*} and Corollary \ref{cor: binom L(T_m)} in (\ref{eq5})
\[ a_{Qm,k} \geq [t^{j^*}]W(T)\binom{L(T_m)}{k-j^*}2^{L(T_m)-k} \geq 2^{L(T_m) -k}.   \]
Hence, by Lemma \ref{lemma: L(T_m) lower bound},
\[ \log a_{Qm,k} = \Omega_{a,\delta}( L(T_m) - k)= \Omega_{a,\delta}(L(T_m)) =  \Omega_{a, \delta}(2^{O(Q)}2^{p(Q,m)m}k^{1-\frac{p(Q,m)}{Q}}). \]

\end{proof}

\section{Extending to all dimensions, and proof of the main Theorems}
In Propositions \ref{prop: upper bound coefficients} and \ref{prop: lower bound} we computed the asymptotics
$\log a_{Qm,k}$. This gives us the asymptotics in all dimensions of the form $2^{Qm}$. Obtaining asymptotics all dimensions which are powers of $2$ will be an easy application of Remark \ref{remark: increasing sequence}.\\
\begin{prop}\label{prop: all dimensions asymptotics}
    Let $n \in \N$ and let $d=2^n$. Decompose $n = Qm+r$ where $m = \lfloor \frac{n}{Q}\rfloor$ and $r \in \{0,1,\ldots,Q-1\}$. Then
    \[ \log a_{Qm,k} \leq \log a_{n,k} \leq 2^{r}(\log k+ \log A_{n,k}) = O(2^r\log A_{n,k}). \]
So, by Propositions \ref{prop: lower bound} and \ref{prop: upper bound coefficients},
\[ \log a_{Qm+r,k}  \sim 2^{O(Q)} 2^{np(Q,m)}k^{1-\frac{p(Q,m)}{Q}}\:,\]
when $k = d^{\delta}$ for some fixed $\delta \in (0,1)$. All the implied constants depend only of $a$ and $\delta$.
\end{prop}

\noindent We now prove the main Theorem.

\begin{proof}[Proof of Theorem \ref{thm: Hanner asymptotics}]
    Let $n \in \N$ and let $d = 2^n$. Fix some $a \in (0,1)$ and let $P_n^a$ be the basic polytopes corresponding to $a$. Decompose $n = Qm +r$ as before. We start with the case $a = \frac{p}{q} \in \Q$. In this case, choose $Q = q$. Then $p(q,m) = p$ and By Proposition  \ref{prop: all dimensions asymptotics}, we have
    \[ \log a_{n,k} \sim 2^{mp}k^{1-\frac{p}{q}}, \]
    In particular, for a fixed $\delta \in (0,1)$,
    \[ \log a_{n,\lfloor d^{\delta}\rfloor } \sim  d^{a+\delta(1-a)}. \]

\noindent for $a \notin \Q$ we choose some $Q=Q_n = o(n)$. By Proposition \ref{prop: all dimensions asymptotics},
\[ \log a_{n,k} \sim 2^{O(Q_n)}2^{p(Q_n,m)m}k^{1-\frac{p(Q_n,m)}{Q_n}} \: ,  \]
we have $\frac{P(Q_n,m)}{Q_n} = a + O(\frac{1}{Q_n})$, and \[ 2^{p(Q_n,m)m}k^{1-\frac{p(Q_n,m)}{Q_n}} = (2^n)^{\frac{m}{n}p(Q_n,m)}k^{1-\frac{p(Q_n,m)}{Q_n}} \sim (2^n)^{a+O(\frac{1}{Q_n})} k^{1-a(1+O(\frac{1}{Q_n}))}.\] 
Plugging $k = \lfloor d^{\delta} \rfloor $ finally gives
\[  \log a_{n,\lfloor d^{\delta} \rfloor} \sim 2^{O(Q_n)} d^{a + \delta (1-a) + O(\frac{1}{Q_n})} = d^{a + \delta(1-a) + o_d(1)}\; , \]
where the $o_d(1)$ term is $d^{O(\frac{Q_n}{n} + \frac{1}{Q_n})}$.
Choosing $Q_n \sim \sqrt{n} \sim \sqrt{\log_2d}$ gives the bound $O(\frac{1}{\sqrt{\log d}})$ on the $o_d(1)$ term.
\end{proof} 

\begin{proof}[Proof of Theorem \ref{thm: FLM}] The proof idea was given in the discussion before Theorem \ref{thm: FLM}. The family of polytopes we take is the sections $E \cap P_n^a \subset \R^{ d-\lfloor d^{\delta} \rfloor}$, where $E$ is a random subspace of $\R^d$ of dimension $d-\lfloor d^{\delta}\rfloor$. The bounds on the number of facets of the section and on the quantity $\left( \frac{R_d}{r_d}\right)^2$ of the section were obtained in \cite{Mil} (with high probability on the choice of the section $E$). The bound on the number of vertices is obtained by   
Theorem \ref{thm: Hanner asymptotics}, as the number of $d- \lfloor d^{\delta} \rfloor$ faces of $P_n^a$ bounds from above the number of vertices of the section (with probability $1$), as was explained in that discussion.
    
\end{proof}

\bibliographystyle{amsplain}

\end{document}